\documentclass[12pt]{amsart}
\long\def\eatit#1{}

\usepackage{pgf,tikz,hyperref,comment}
\usepackage{amsmath, graphicx}
\usepackage{amssymb,color}
\usepackage{enumitem}
\usepackage{comment}

\newtheorem{theorem}{Theorem}[section]
\newtheorem{proposition}[theorem]{Proposition}
\newtheorem{lemma}[theorem]{Lemma}
\newtheorem{corollary}[theorem]{Corollary}
\newtheorem{conjecture}[theorem]{Conjecture}
\newtheorem*{theorem*}{Theorem}

\newtheorem{assump}{Theorem}

\theoremstyle{definition}

\newtheorem{prob*}{Problem}
\newtheorem{question}[theorem]{Question}
\newtheorem{definition}[theorem]{Definition}

\newtheorem{example}[theorem]{Example}
\newtheorem{remark}[theorem]{Remark}

\newtheorem{thevarthm}[theorem]{\varthmname}

\newtheorem{thmx}{Theorem}

\newenvironment{varthm*}[1]{\trivlist\item[]{\bf #1.}\it}{\endtrivlist}

\newcommand{\coker}{ \ensuremath{coker}}
\newcommand{\PP}{ \ensuremath{\mathbb{P}}}

\thanks{\hspace{-15pt} 
{\it Acknowledgment.} Favacchio thanks the University of Notre Dame for its hospitality and the support of  the Università degli studi di Palermo “Piano straordinario per il miglioramento della qualità della ricerca e dei risultati della VQR
2020-2024 - Misura A” and GNSAGA-INdAM. 
Migliore was partially supported by Simons Foundation grant \#839618.
\
This work was begun during the Preliminary School for the  conference on Lefschetz Properties in Algebra, Geometry, Topology and Combinatorics. This school took place in Krak\'ow on May 5-11, 2024, and the  authors are grateful for the kindness and support of the organizers.}

\begin{document}

\begin{abstract}
Ideals $I\subseteq R=k[\mathbb P^n]$ generated by powers of linear forms arise, via Macaulay duality, from sets of fat points $X\subseteq \mathbb P^n$. Properties of $R/I$ are connected to the geometry of the corresponding fat points. When the linear forms are general, many authors have studied the question of whether or not $R/I$ has the Weak Lefschetz Property (WLP). We study this question instead for ideals coming from a family of sets of points called grids. We give a complete  answer in the case of uniform powers of linear forms coming from square grids, and we give a conjecture and approach for the case of nonsquare grids. In the cases where WLP holds, we also describe the non-Lefschetz locus.

\end{abstract}


\title[WLP for powers of linear forms]{On the Weak Lefschetz Property for certain ideals generated by powers of linear forms}

\date{June 7, 2024}
\setcounter{tocdepth}{2} 

\author[G.~Favacchio]{Giuseppe Favacchio}
\address[G.~Favacchio]{Dipartimento di Ingegneria, Universit\`a degli studi di Palermo,
Viale delle Scienze,  90128 Palermo, Italy}
\email{giuseppe.favacchio@unipa.it}

\author[J.~Migliore]{Juan Migliore} 
\address[J.~Migliore]{Department of Mathematics,
University of Notre Dame,
Notre Dame, IN 46556 USA}
\email{migliore.1@nd.edu}

\maketitle

\section{Introduction}

Many papers have studied the question of the Weak Lefschetz Property (WLP) or Strong Lefschetz Property (SLP) (or their failure) for quotients of polynomial rings by ideals generated by powers of linear forms. A partial list is \cite{BL}, \cite{POLITUS1}, \cite{POLITUS4}, \cite{HSS}, \cite{HMNT}, \cite{MM}, \cite{MMN}, \cite{MN}, \cite{MR}, \cite{MT}, \cite{NT}, \cite{SS1} and \cite{SS2}.

Almost all of these papers have used inverse systems and the connection to ideals of fat points in different ways, and almost all have focused on ideals generated by powers of {\it general} linear forms (often with uniform powers). Many also restricted to the situation of {\it almost complete intersections}, i.e. the case where the number of minimal generators is one more than the number of variables.

The paper \cite{SS1} showed that in a polynomial ring with 3 variables, an {\it arbitrary} such algebra has the WLP. \cite{BL} settled a conjecture of \cite{MMN}. \cite{MR} settled a conjecture of \cite{HSS}.
The paper \cite{HMNT} formally made the connection to so-called {\it unexpected hypersurfaces}  (see \cite[Proposition 2.17]{HMNT}), and \cite{POLITUS1} connected it to the geproci property.

Recall from \cite{POLITUS1} that a set of reduced points $X$ in $\mathbb P^3$ is  {\it $(a,b)$-geproci} if its general projection to $\mathbb P^2$ is a complete intersection of type $(a,b)$. In particular, a special type of $(a,b)$-geproci set is an $a \times b$ grid, i.e.  a set of points lying on a smooth quadric surface and given by the intersection of $a$ lines in one ruling and $b$ lines in the other. 

In this  paper we  continue the study of ideals coming from  sets of points with interesting geometry. We start with a grid $X$, and use the fact that a general projection is a complete intersection to derive information about the Weak Lefschetz Property for quotients of ideals generated by any power of the linear forms dual to the points of $X$. These are, of course, far from general forms.
 Let $\Lambda_{X,d}= (\ell_{i,j}^d\ |\ 1\le i,j\le a )$ be the ideal generated by the $d$-th powers of the linear forms dual to the points in the set $X$.  
 From \cite[Corollary 2.30]{POLITUS1}, which is a consequence of \cite[Theorem 3.5]{CM}, it is known that: {\it If $X$ is an $a\times b$ grid with $b \geq a \geq 2$ and $b \geq 3$ then $R/\Lambda_{X,a}$ fails the WLP from degree $a-1$ to degree $a$, and if $b \geq a \geq 3$ then $R/\Lambda_{X,b}$ fails the WLP from degree $b - 1$ to degree $b$.}

In Section \ref{sec. Gor}, we show that given an $a\times b$ grid $X$ and $d\le a-1$, then $R/\Lambda_{X,d}$ is a compressed Gorenstein algebra of even socle degree, 
Theorem \ref{l. in a grid is enough}. Thus, as a consequence, we get that it has the WLP. 
In Section \ref{sec. coker geproci} and Section \ref{sec. hom inv grids} we prove some preparatory results that we will apply in the study of the WLP of $R/\Lambda_{X,d}$ for larger $d$. In particular, in Lemma \ref{l. cokernel (a,b)-geproci}, given $X$ an $(a,b)$-geproci and $\ell$  a general linear form,
we compute the dimension of the cokernel of the map $\times \ell :A_{d+t-1}\to A_{d+t}$ for $t\ge 0$. In Corollary \ref{c. no syzygies in low degrees} we show that $R/\Lambda_{X,d}$ has no syzygies in low degrees. 
The main results of Section~\ref{sec. WLP square grids} are summarized in the following statement.

\begin{thmx}
      Let $X$ be an $a\times a$ grid, $a\ge 3,$  and let $I$ be the ideal generated by $\ell^d$ for each $\ell^{\vee}\in X$. Then 
$R/I$ has the WLP if and only if 
        either $d\le a-1$  [Corollary \ref{low powers, square grid}] or  $a-1$ divides $d$
         [Theorem \ref{t. WLP fails} and  Theorem \ref{t. WLP holds}].
\end{thmx}

\noindent In Section \ref{sec: non-Lefschetz locus} we find the non-Lefschetz locus for the algebras $R/\Lambda_{X,d}$ which according to Theorem A have the WLP (conjecturally they are all such algebras with WLP). 
In Section~\ref{sec. conj non square grids} we pose some questions and conjectures coming from this paper. 

\section{Background and basic results}

Let $R = k[x_1,\dots,x_4]$ be the homogeneous polynomial ring in four variables with the standard grading, where $k$ has characteristic zero.

\begin{definition}
    Let $A$ be an artinian graded $k$-algebra. Let $\ell$ be a general linear form. We say that $A$ has the {\it Weak Lefschetz Property (WLP)} if $\times \ell : [A]_{t-1} \rightarrow [A]_t$ has maximal rank for all $t$, i.e., it is either injective or surjective.
\end{definition}

The following result is standard, and is an important tool in our  computations.

\begin{lemma} \label{pwrs of ci}
    Let $I \subset R = k[\mathbb P^2]$ be a complete intersection.  Then the powers $I^m$ are equal to the symbolic powers $I^{(m)}$ for all $m$. Furthermore, if $I$ has minimal free resolution 
    \[
    0 \rightarrow \mathbb F \rightarrow \mathbb G \rightarrow I \rightarrow 0
    \]
    (with $\hbox{rk } \mathbb F = 1$ and $\hbox{rk } \mathbb G = 2$) then $I^m$ has minimal free resolution
    \[
    0 \rightarrow \mathbb F \otimes \hbox{Sym}^{m-1} \mathbb G \rightarrow \hbox{Sym}^m \mathbb G \rightarrow I^m \rightarrow 0.
    \]
\end{lemma}

\begin{proof}
    See for instance \cite{codim2} Lemma 2.8 and Theorem 2.9.
\end{proof}

\begin{lemma}
    Let $X$ be a set of $t$ collinear points in $\mathbb P^n$, $X = \{ P_{1},\dots,P_{t} \}$, and let $\ell_{i}$ be the linear form dual to $P_{i}$ for each $i$. Let $d$ satisfy $ d \leq t-1$. Then 
    \[
  \dim_k  \langle \ell_{1}^d,\ldots, \ell_{t}^d \rangle =d+1. \ 
    \]
In particular any $d+1$ of the forms $\ell_i^{d}$ can be chosen as a minimal basis of the vector space spanned by all the $\ell_i^d$.  
\end{lemma}
\begin{proof}
    Consider the ideal $(\ell_1^d,\dots,\ell_t^d)$. We have using Macaulay  duality
    \[
 \dim_k  \langle \ell_{1}^d,\ldots, \ell_{t}^d \rangle   =  \dim [( \ell_{1}^d,\ldots, \ell_{t}^d )]_d=\dim [R/I_X]_d =d+1
    \]
     since $t \geq d+1$.
\end{proof}

\begin{corollary}
    Let $X$ be an $a \times b$ grid, with $a \leq b$ and $X = \{ P_{1,1},\dots,P_{a,b} \}$, and for each $i,j$ let $\ell_{i,j}$ be the linear form dual to $P_{i,j}$.  Then 

    \begin{itemize}
        \item[$(a)$] If $ a-1 \leq d \leq b-1$ then  for any  $a \times (d+1)$ subgrid $Y$ of $X$  we have an equality of ideals
    \[
    (\ell_{i,j}^d \ | \ \ell^{\vee}\in X  ) = (\ell_{i,j}^d \ | \ \ell^{\vee}\in Y ). \ 
    \]
        \item[$(b)$] If $d \leq a-1$  then  for any  $(d+1) \times (d+1)$ subgrid $Y$ of $X$  we have an equality of ideals
    \[
    (\ell_{i,j}^d \ | \ \ell^{\vee}\in X ) = (\ell_{i,j}^d \ | \  \ell^{\vee}\in Y ). \ 
    \]
    \end{itemize}

\end{corollary}

\begin{remark}\label{r.P1xP1}
We recall some basic facts on bigraded rings and points in $\mathbb P^1\times \mathbb P^1$; see for instance~\cite{GV-book} for an overview on the topic.
    Let $S:=k[x_0,x_1,y_0,y_1]$ be a polynomial ring with coefficients in $k$ and set $\deg(x_i)=(1,0)$ and $\deg(y_i)=(0,1),$ for $i=0,1.$  For each $(i,j)$ non-negative integers $S_{i,j}$ denotes the vector space generated by the monomials of bi-degree $(i,j)$. 
For a bi-homogeneous ideal $I\subseteq S$ the Hilbert function of $S/I$ is the numerical function $h_{S/I}:\mathbb Z^2\to \mathbb Z$ defined by $$h_{S/I}(i,j)=\dim \left[S/I\right]_{i,j}=\dim \left[S\right]_{i,j}-\dim \left[I\right]_{i,j}.$$
The first difference of $h_{S/I}$ is defined as
$$\Delta h_{S/I}(i,j)=h_{S/I}(i,j)-h_{S/I}(i-1,j)-h_{S/I}(i,j-1)+h_{S/I}(i-1,j-1).$$

A point in $\mathbb P^1\times \mathbb P^1$ is a pair $A\times B$ where $A,B$ are both points in $\mathbb P^1$, but $A$ is defined by a linear form in $k[x_0,x_1]$ and $B$ is defined by a linear form in $k[y_0,y_1]$.
For a set of points $Z$ in $\mathbb P^1\times \mathbb P^1$ the ideal $I_Z$ is bihomogeneous. 
Many results, in analogy to the standard graded case, can be proved for reduced and non-reduced sets of points in $\mathbb P^1\times \mathbb P^1$. For a (not necessarily reduced) set of points $Z$ in $\mathbb P^1 \times \mathbb P^1$ the function $\Delta h_{Z}$ has only a finite number of nonzero entries and, in particular, if the ideal $I_Z$ is non-zero in bidegrees $(u,0)$ and $(0,v)$ then $\Delta h_{Z}(i,j)=0$ for either $i\ge u$ or $j\ge v$.  (See Section 2 in \cite{GMR} for these and more properties on $h_{Z}$.)
\end{remark}

\section{A class of Gorenstein ideals generated by powers of linear forms}\label{sec. Gor}

The results of this section are of independent interest. We study a class of Gorenstein algebras that are generated by powers of linear forms. However, as a byproduct we obtain our first result (Corollary \ref{low powers, square grid}) about WLP for certain instances of our main topic, namely the question of WLP or its failure for ideals generated by powers of linear forms that are dual to grid points on a smooth quadric. In this case, we handle the situation where the linear forms have ``small" powers.

Let $\mathcal Q$ be a smooth quadric defined by a quadratic form $Q$, and let $\ell \in R_1$. Given two forms $F,G\in R$ the symbol $\dfrac{\partial F}{\partial G}$ denotes the derivation of $F$ under the action of $G.$
We have the following preliminary lemmas. 

\begin{lemma}\label{l. P in Q}
    Let $\ell\in R_1$ and $P=\ell^{\vee}.$ Then
    $P\in \mathcal Q$  if and only if  $\dfrac{\partial  Q}{\partial \ell^2}=0$. 
\end{lemma}
\begin{proof}
   Let $\ell=\sum \alpha_ix_i$ and $Q=\sum_{i\le j} a_{ij}x_ix_j$. Then,
    $\dfrac{\partial Q}{\partial \ell^2}=2\sum_{i\le j}a_{ij}\alpha_i\alpha_j =0 $ if and only if $ P\in \mathcal Q.$
\end{proof}

\begin{lemma}\label{l. in a grid is enough}
    Let $G$ be an $a\times a$ grid in $\mathcal Q$. Then 
    $$(\ell^{a-1}\ |\ \ell^{\vee}\in \mathcal Q )=(\ell^{a-1}\ |\ \ell^{\vee}\in G ).$$
    In particular $\dim[(\ell^{a-1}\ |\ \ell^{\vee}\in \mathcal Q )]_{a-1}=a^2.$
\end{lemma}
\begin{proof}
Since $(\ell^{a-1} \ | \ \ell^\vee \in G )  \subseteq ( \ell^\vee \ | \ \ell^\vee \in \mathcal 
Q )$, it is enough to show that they have the same dimension in degree $a-1$. Note that by  Bezout's theorem, any  form of degree $a-1$ containing the points of $G$ must also contain the grid lines, and since the union of the grid lines is a complete intersection of type $(2,a)$, any  such form must be a multiple of $Q$. 
Then by Macaulay duality, the degree $a-1$ component of either ideal has dimension equal to $\dim [R/(Q)]_{a-1}$, which is $a^2$.
\end{proof}
Now we prove the main result of this section.  Given a form $F\in R$, $F^{\perp}$ denotes the ideal $F^{\perp}=\left(G\in R\ |\ \dfrac{\partial F}{\partial G}=0\right)$. It is known that $R/F^{\perp}$ is an artinian Gorenstein algebra.  See for instance \cite{I.compressed, IK} for an in-depth analysis of such algebras.

\begin{lemma}\label{l. hope so}Let $\mathcal Q$ be a smooth hypersurface in $\mathbb P^n$ of degree $r\ge 2$ defined by the form $Q.$ We have  
     $[(Q^t)^{\perp}]_t=(0)$.
\end{lemma}
\begin{proof}
We claim that the elements in the set
$$\left\{\dfrac{\partial Q^t}{\partial M}\ |\ M\  \text{monomial in }\ R_t\right\}$$
are linearly independent for any $t\ge 0$. 
The case $t=1$ is true since $\mathcal Q$ is smooth. Let $t>1$.

Thus let $\alpha_M\in k$ be such that
$$\sum_M \alpha_M\dfrac{\partial Q^t}{\partial M}=0.$$
This implies (set $c_M$ the exponent of $x_1$ in $M$ 
)
$$\sum_{M\ :\ c_M>0 } \alpha_M\dfrac{\partial Q^t/ \partial x_1}{\partial M/ \partial x_1}=\sum_{N\ :\ M=x_1N} \alpha_Mt\dfrac{\partial Q}{\partial x_1}\dfrac{\partial Q^{t-1}}{\partial c_M N}=t\dfrac{\partial Q}{\partial x_1}\cdot \sum_{N\ :\ M=x_1N} c_M\alpha_M\dfrac{\partial Q^{t-1}}{\partial N}=0.$$
Since the monomials $N$ have degree $t-1$, by induction we get $c_M\alpha_M=0$ for any $c_M>0$, that is $\alpha_M=0$ for any $M$ such that $x_1|_M$. Repeating the same procedure above for the other variables we conclude the proof of the  claim. But then the lemma follows since if $G$ is a form of degree $t$ such that $\frac{\partial Q^t}{\partial G} = 0$, one simply writes $G$ as a linear combination of monomials and applies the claim.
\end{proof}

\begin{theorem}\label{Gor gen}Let $\mathcal Q$ be a smooth quadric surface  in $\mathbb P^3$ defined by $Q\in R_2.$  We have
    $$(Q^t)^{\perp}= (\ell^{t+1}\ |\ \ell^{\vee}\in \mathcal Q )$$
    and $R/(Q^t)^\perp$ is a compressed Gorenstein algebra.
\end{theorem}
\begin{proof}
    Set $I=(\ell^{t+1}\ |\ \ell^{\vee}\in \mathcal Q ).$
First we show  that $I\subseteq (Q^t)^{\perp}.$
Let $P\in \mathcal Q$ and set $\ell=P^{\vee}.$
The case $t=1$ is Lemma \ref{l. P in Q}.
So assume $t\ge 1$. We have
 \[
\begin{array}{rl}
   \dfrac{\partial Q^{t}}{\partial \ell^{t+1}}= &\displaystyle \sum_{i=0}^{t+1}\binom{t+1}{i}\dfrac{\partial Q^{t-1}}{\partial \ell^i}\dfrac{\partial Q}{\partial \ell^{t+1-i}}  \\[11pt]
     =&\displaystyle \sum_{i=0}^{t-1}\binom{t+1}{i}\dfrac{\partial Q^{t-1}}{\partial \ell^i}\underbrace{\dfrac{\partial Q}{\partial \ell^{t+1-i}}}_{=0\ \text{by\ Lemma}\  \ref{l. P in Q}} +\sum_{i=t}^{t+1}\binom{t+1}{i}\underbrace{\dfrac{\partial Q^{t-1}}{\partial \ell^i}}_{=0 \ \text{by induction}}\dfrac{\partial Q}{\partial \ell^{t+1-i}}=0.\\
\end{array} \] 

Now we show that $I\supseteq (Q^t)^{\perp}.$ 
 Since $R/(Q^t)^\perp$ has socle degree $2t$ and the ideal starts in degree $t+1$ (from Lemma \ref{l. hope so}), we have that $R/(Q^t)^\perp$  is a compressed Gorenstein algebra. Finally, one checks that in degree $t+1$ this forces the dimension of the ideal to be $(t+2)^2$, which in Lemma \ref{l. in a grid is enough} we saw is the case for $I$. Therefore $I = (Q^t)^\perp$.
\end{proof}

\begin{corollary} \label{low powers, square grid}
    Let $X \subset \mathcal Q$ be an $a \times b$ grid with $a \leq b$ and let $\ell_1,\dots,\ell_{ab}$ be the dual linear forms. Let $1 \leq d \leq a-1$ and let $I = (\ell_1^d, \dots, \ell_{ab}^d)$. Then $R/I$ has the WLP.
\end{corollary}

\begin{proof}
    Let $Y \subset X$ be a $(d+1) \times (d+1)$ subgrid of $X$. Consider the ideal $(Q^{d-1})^\perp$. By Lemma \ref{l. in a grid is enough} and Theorem \ref{Gor gen}, $I=(\ell^d\ |\ \ell^{\vee} \in Y)$ and thus  $R/I$ is a compressed Gorenstein algebra with even socle degree. Therefore $R/I$ has the WLP.
\end{proof}

\section{The dimension of the cokernel and geproci sets}\label{sec. coker geproci}

For a finite set of points $X\subseteq \mathbb P^3$ and a positive integer $d$, we denote by $\Lambda_{X,d}$ the ideal generated by the $d$-th powers of the linear forms dual to the points in $X$.

In the next lemma we apply some known results to the case of geproci sets. It is crucial for this work since it allows us to compute the dimension of the cokernel of the map  $\times \ell$ in a certain range.

\begin{lemma}\label{l. cokernel (a,b)-geproci} Let $X$ be an $(a,b)$-geproci set in $\mathbb P^3$. Let $P$ be a general point in $\mathbb P^3$ and $\ell$ be the linear form dual to $P$. Let  $\overline{X}=\pi_P(X)$ be the projection of $X$ from $P$ to a general plane. Set $I=\Lambda_{X,d}$ and $A=R/I$. 

Then, if $t$ is a non-negative integer, for the dimension of the cokernel of the map 
$$\times \ell:A_{d+t-1}\to A_{d+t}$$
 we have:
\begin{itemize}
    \item[$(a)$]$\dim coker(\times \ell)=\dim[I_{\overline X}^{t}]_{d+t};$
    \item[$(b)$] moreover, if $a(t+1)+b>d +t$   then 
    $$\dim coker(\times \ell)=\sum_{i=0}^{t+1}\binom{d+t+2-a(t+1)-(b-a)i}{2}.$$
\end{itemize}
\end{lemma}
\begin{proof}
Consider the exact sequence
\begin{equation} \label{std seq}
    R/(I)(-1) \stackrel{\times \ell}{\longrightarrow} R/(I) \rightarrow R/(I+ (\ell)) \rightarrow 0.
\end{equation}
Let $\wp$ be the ideal of the point $P$ dual to $\ell$ in $\PP^3$. We have, using Macaulay duality, 
\[
\dim [R/(I+(\ell))]_{d+t} = \dim[I_{X}^{(t+1)}\cap \wp^{d+t}]_{d+t}.
\]
By \cite[Proposition 3.4]{MMN}, $$\dim[I_{X}^{(t)}\cap \wp^{d+t}]_{d+t}= \dim [I_{\overline{X}}^{(t+1)}]_{d+t},$$ 
where the vector space on the left is contained in  $k[\mathbb P^3]_{d+t}$ and the one on the right in $k[\mathbb P^2]_{d+t}$.  
    
Since $X$ is an $(a,b)$-geproci, therefore $\overline{X}$ is a complete intersection of type $(a,b)$, i.e., $I_{\overline{X}}$ is generated by a regular sequence $(f,g)$ for $S=k[\mathbb P^2]$, with $f\in S_a$ and $g\in S_{b}$. As such, any power is a saturated ideal, i.e., $I_{\overline{X}}^{(t+1)}=I_{\overline{X}}^{t+1}=(f,g)^{t+1}$ for any $t\ge 0$ (see for instance \cite{GV} Lemma 4.1, \cite{ZS} Lemma 5, Appendix 6 or \cite{codim2} Theorem 2.9). Thus $$\begin{array}{rl}
     \dim(coker(\times \ell))=& \dim [R/(I,\ell)]_{d+t} = \dim [I_{\overline{X}}^{t}]_{d+t} \\
\end{array}.$$ 
This proves $(a)$.

For $(b)$ notice that, since the first syzygies for the ideal $(f,g)^{t+1}=(f^{t+1},f^{t}g,\cdots )$ only begin in degree $a(t+1)+b$ (see Lemma \ref{pwrs of ci}), and we are assuming $a(t+1)+b>d +t$, thus we get
$$[I_{\overline{X}}^{t+1}]_{d+t}=
\bigoplus_{i=0}^{t+1} [f^{t+1-i}g^i]_{d+t}=\bigoplus_{i=0}^{t+1} \left[S\right]_{d+t-a(t+1-i)-bi}.$$
And then

\[\dim coker(\times \ell)=\dim[I_{\overline{X}}^{t+1}]_{d+t}=
 \sum_{i=0}^{t+1}  \binom{d+t+2-a(t+1)-(b-a)i}{2}.\qedhere\]
\end{proof}

\section{Some homological invariants related to $a\times b$ grids}\label{sec. hom inv grids}
In this section $X$ denotes an $a \times b$ grid on a smooth quadric surface $\mathcal Q$, with $b\ge a \ge 2$. 

 Let $I_X$ be the homogeneous ideal of $X$ and $I_X^{(d)} = (I_X^d)^{sat}$ its $d$-th symbolic power, for any $d \geq 1$. Set $X = \{ P_{i,j}|\ 1\le  i\le a, 1\le j\le b \}$ and let $X^{(d)}$ be the scheme defined by $I_X^{(d)}$.

 For $d \geq 2$ the scheme $X^{(d)}$ does not lie on $Q$, but we denote by $Y_d$ the subscheme that does lie on $Q$, that is $I_{Y_d} = (I_X^{(d)} + (Q))^{sat}$ as a subscheme of $\PP^3$ and $I_{Y_d|Q}$ is defined by
\[
\left ( \frac{I_X^{(d)} + (Q)}{(Q)} \right )^{sat}.
\]

\begin{lemma} \label{l. indep cond}
Let $X$ be an $a\times b$ grid, say $b\ge a\ge 2$, and let $d$ be a positive integer.   Then  
	
	\begin{itemize}
		\item[$(a)$] The value of the Hilbert function of $Y_d$ (as a subscheme of $\PP^3$) in degree $b d-1$ is equal to the multiplicity of $Y_d$. That is,
		\[
		h_{{Y_d}} (b d-1) = ab\cdot \binom{d+1}{2}.
		\]
		In other words, as a subscheme of $\PP^3$, $Y_d$ imposes independent conditions on forms of degree $b d-1$.

		\item[$(b)$] For any $d$, the value of the Hilbert function of $X^{(d)}$ in degree $b d-1$ is equal to the multiplicity of $X^{(d)}$. That is,
		\[
		h_{X^{(d)}} (b d-1) = ab \cdot \binom{d+2}{3}.
		\]
		In other words, $X^{(d)}$ imposes independent conditions on forms of degree $bd-1$.
		
	\end{itemize}
\end{lemma}
\begin{proof}
We first prove $(a).$ Let $Z$ be a set of points in $\mathbb P^1\times \mathbb P^1$ whose Segre embedding in  $\PP^3$ is $X$.  	Note that $Z$ is a codimension 2 complete intersection in $\mathbb P^1\times \mathbb P^1$ generated by forms of bi-degree $(a,0)$ and $(0,b)$. Hence, from  \cite[Lemma 5, Appendix 6]{ZS},  $I_Z^{(d)}=I_Z^{d}$.   

Let $Z_d$ be the zero-dimensional scheme in $\mathbb P^1\times \mathbb P^1$ defined by $I_Z^{(d)}$. By \cite[Remark 2.1]{GMR}, we have
	 $$h_{Y_d}(t)= h_{Z_d}(t,t), \ \text{for any}\ t\ge 0. $$
Since the ideal of $Z_d$ has a minimal generator in degrees $(a d,0)$ and one in $(0, b d)$ then, $\Delta h_{Z_d}(i,j)=0$ if $i\ge ad$ or $j\ge bd$ (see Remark \ref{r.P1xP1}). Since  
$h_{Z_d}(u,v)=\sum\limits_{i\le u,\ j\le v} \Delta h_{Z_d}(i,j)$,   see \cite[Remark 2.8,Remark 2.1]{GMR}, then
for any $t\ge b d-1$ we get
$$h_{Z_d}(t,t)=\deg (Z_d)=ab \cdot \binom{d+1}{2}.$$ 

We now prove $(b)$,  by induction on $d$. The statement is true for $d=1$.
Indeed, for any point $P_{i,j}\in X$ we can construct a form $F$ of degree $b-1$ vanishing in $X\setminus\{P_{i,j}\}$ and not in $P_{i,j}.$
We show this construction for $P_{1,1}$, but the general case is completely analogous.
For $i=2,\ldots, a$ let $H_i$ be the linear form defining the plane spanned by the two rulings containing $P_{i,i}$, and for $i=a+1,\ldots, b$ let $H_i$ be the linear form defining a general plane through the point $P_{1,i}.$  
Then $F=H_2\cdots H_a\cdot H_{a+1}\cdots H_{b}$ is a form of degree $b-1$ which does not vanish in $P_{1,1}$ and it vanishes in $X\setminus\{P_{1,1}\}$ since $H_i(P_{i,j})=0$ if either $i>2$ or $i=1$ and $j\le a$
otherwise $H_j(P_{i,j})=0$ for $j> a$.

Now assume that $(b)$ is true for $I_X^{(t)}$ for all $t \leq d-1$. 

Consider the exact sequence
\[
0 \rightarrow (I_X^{(d)} : (Q)) (-2) \stackrel{\times Q}{\longrightarrow} I_X^{(d)} \rightarrow I_{X}^{(d)}/Q \cdot (I_X^{(d)} : (Q)) \rightarrow 0.
\]
Since $Q$ is smooth, it vanishes to multiplicity one at each point of $X$ so we have 
\[
I_X^{(d)} : (Q) = I_X^{(d-1)}.
\]
Furthermore, 
$Q \cdot (I_X^{(d)} : (Q)) = I_X^{(d)} \cap (Q)$ and 
\[
\frac{I_X^{(d)} }{ I_X^{(d)} \cap (Q) } \cong \frac{I_X^{(d)} + (Q)}{(Q)}.
\]
Thus
\begin{equation} \label{new ses}
0 \rightarrow I_X^{(d-1)} (-2) \stackrel{\times Q}{\longrightarrow} I_X^{(d)} \rightarrow \frac{I_X^{(d)} + (Q)}{(Q)} \rightarrow 0   
\end{equation}

is exact.

From the sequence (\ref{new ses}) we sheafify and twist to get
\begin{equation} \label{sheafseq1}
0 \rightarrow \mathcal I_{X^{(d-1)}} (b d-3) \rightarrow \mathcal I_X^{(d)} (b d-1) \rightarrow \mathcal I_{Y|Q} (b d-1) \rightarrow 0
\end{equation}
where the last sheaf is the ideal sheaf of $Y$ viewed as a subscheme of $Q$ and the last map is restriction to $Q$. 
Note also the short exact sequence of sheaves, for any integer $t$, relating the properties of $Y_d$ as a subscheme of $\PP^3$ and as a subscheme of $Q$:
\begin{equation} \label{sheafseq2}
0 \rightarrow \mathcal O_{\PP^3} (t-2) \rightarrow \mathcal I_Y (t) \rightarrow \mathcal I_{Y|Q} (t) \rightarrow 0. 
\end{equation}
It follows from (\ref{sheafseq2}) and from part (a) that 
\[
h^1(\mathcal I_Y (t)) = h^1(\mathcal I_{Y|D}(t)) = 0 \ \ \hbox{ for } \ \ t \geq b d-1.
\]
We also have
\[
h^1(\mathcal I_X^{(d-1)} (b d-3))  =  0
\]
because by induction, $R/I_X^{(d-1)}$ reaches its multiplicity by degree $b  ( d-1)-1$ so certainly it is also equal to the multiplicity in degree $b d-3$. 
Then from the sequence (\ref{sheafseq1}) we obtain
\[
h^1(\mathcal I_X^{(d)}(b d-1)) = 0
\]
and so $X^{(d)}$ imposes independent conditions on forms of degree $b d-1$.
\end{proof}

Let $\Lambda_{X,d}= (\ell_{i,j}^d\ |\ 1\le i,j\le a )$  the ideal generated by the $d$-th powers of the linear forms dual to the points in $X$ and  $A_{X,d} = R/\Lambda_{X,d}$  for a positive integer $d$. 
In the next result, using Macaulay duality and Lemma \ref{l. indep cond}, we derive information for the ideal $\Lambda_{X,d}$.

\begin{corollary}\label{c. no syzygies in low degrees}
	Let $X$ be an $a\times b$ grid and set $I=\Lambda_{X,d}$.  Let $d\ge b-1$ be any integer. Denote by $q,r$ the non-negative integers such that  $d= (b-1) q+ r$ and $0 \leq r< b-1$ (note that we must have $q\ge 1$). Then
	\begin{itemize}
		\item[(a)] $ [I]_{t} = [R]_{t-d} \cdot [I]_d$ for $d \leq t \leq d+q-1$. In particular,  $\dim [I]_t = ab\cdot  \displaystyle \binom{t-d+3}{3}$ for $d \leq t \leq d+q-1$;
		
		\item[(b)] $I$ has no syzygies of degree $\leq q-1$;
		
		\item[(c)] $R/I$ has no socle in degree $\leq d+q-2$.
	\end{itemize}
\end{corollary}
\begin{proof}
	It is enough to prove it for $t =d+ q-1$. By Macaulay duality  we have 
	\[
	\begin{array}{rcl}
		\displaystyle \dim [\Lambda_{X,d}]_{d+q-1} & = & \dim [R/I_X^{(  q )}]_{d+q-1}=\dim [R/I_X^{(  q )}]_{bq+r-1}. \\
\end{array}
	\]
Since   $bq+r-1 \ge bq-1$, by Lemma \ref{l. indep cond}  $\dim [R/I_X^{(  q )}]_{bq+r-1}$ is equal to $\deg X^{( q )}$
thus 
  $$\displaystyle \dim [\Lambda_{X,d}]_{d+q-1}= \displaystyle \deg X^{( q )} = \displaystyle ab\cdot \binom{ q+2 }{3}.$$
	
	But $\binom{q+2}{3} = \dim [R]_{q-1}$ so we are done (a). Then (b) follows immediately from (a), and (c) follows from (b) since then the upper part of the Betti diagram for $R/I$ looks like
	\[
	\begin{array}{c|ccccc}
		R/I &0&1&2&3&4\\ \hline
		0 & 1 & - & - & -& -  \\
		1 & - & - & - & -& - \\[-5pt]
		\vdots &&&&& \\[-5pt]
		d-2 & - & - & - & -& - \\
		d-1 & - & ab & - &-& -\\
		d & - & - &  - &-& -\\[-5pt]
		\vdots && \\[-5pt]
		d+q-2 & - & - & -& -& - \\
	\end{array}.
	\]
	Thus, in the above table the only non-zero entry corresponds to $\beta_{1,d}(R/I)$ and, in particular, the socle vector must be zero at least until degree $d+q-2$.
\end{proof}

\section{The WLP and $a\times a$ grids}\label{sec. WLP square grids}
In this section $X$ denotes an $a \times a$ grid on a smooth quadric surface $\mathcal Q$. Let $\Lambda_{X,d}= (\ell_{i,j}^d\ |\ 1\le i,j\le a )$ be the ideal generated by the $d$-th powers of the linear forms dual to the points in $X$ and  $A_{X,d} = R/\Lambda_{X,d}$  for a positive integer $d$.  Notice that when $a=2$ we have powers of four general linear forms, which after a change of variables can be chosen to be $x_1, x_2,x_3,x_4$. Thus $R/\Lambda_{X,d}$ automatically  has even the Strong Lefschetz Property, by the famous result of R. Stanley \cite{stanley} and J. Watanabe \cite{watanabe}, but in particular it has the  WLP. So without loss of generality from now on we assume $a \geq 3$.

In this section we show  (Theorem \ref{t. WLP fails} and Theorem \ref{t. WLP holds}) that for $X$  an $a \times a$ grid and $d\ge a-1$: 
$A_{X,d}$ has the WLP if and only if $d$ is multiple of $a-1$.  

We will use connections with the results of the previous sections.

\begin{theorem}\label{t. WLP fails}
Let $X$ be an $a\times a$ grid and $d\ge a-1$. 
Let $q,r$ be the non-negative integers such that $d = (a-1)\cdot q+r$ with $0\le r<a-1$. Set $I=\Lambda_{X,d}$ and $A = R/I$. Then

    \begin{itemize}

    \item[$(a)$] 
   $\displaystyle \Delta h_A(d+q-1) < (q+1)\cdot \binom{r+1}{2}$. 

    \item[$(b)$] For a general linear form $\ell$,   the multiplication map $\times \ell: [A]_{d+q-2} \stackrel{ }{\longrightarrow} [A]_{d+q-1}$ has $$\displaystyle \dim{\coker(\times \ell)}=(q+1)\cdot \binom{r+1}{2}\ge 0.$$

    \end{itemize}
     Therefore, if $r>0$ then $\times \ell$ fails to be injective and $A$ fails to have the WLP.
\end{theorem}
\begin{proof}
We note that the last assertion of the theorem is a consequence of items $(a)$ and $(b)$. Indeed, from $(a)$ the expected dimension of the cokernel of the map $\times \ell $ is strictly smaller than $(q+1)\cdot \binom{r+1}{2}$. From $(b)$, the dimension is in fact $(q+1)\cdot\binom{r+1}{2}$. Hence, for $r>0$, we expect injectivity (since $\dim{\coker(\times \ell)}>0$) but it is not present, so $A$ fails the  WLP.

Now we prove $(a)$. From Corollary \ref{c. no syzygies in low degrees}, we have $$\begin{array}{rl}
     \Delta h_A(d+q-1)=&h_A(d+q-1)- h_A(d+q-2)\\[10pt]
     =& \displaystyle\dim [R]_{d+q-1}-\dim [R]_{d+q-2}- \dim [I]_{d+q-1}+\dim [I]_{d+q-2} \\[10pt]
     =& \displaystyle\binom{d+q+1}{2}- a^2\cdot   \binom{q+1}{2}=\displaystyle \binom{aq+r+1}{2}- a^2\cdot   \binom{q+1}{2}\\[11pt]
     
    =& \displaystyle \dfrac{(aq+r+1)(aq+r)- a^2(q+1)q}{2}\\[11pt]

=& \displaystyle \dfrac{2aqr+aq+r^2+r -a^2q}{2}.\\
\end{array} $$

Moreover, the following inequality
$$\dfrac{2aqr+aq+r^2+r -a^2q}{2}< \dfrac{qr^2+qr+r^2+r}{2}=(q+1)\cdot \binom{r+1}{2}$$ 
holds if and only if
$$2aqr+aq -a^2q< qr^2+qr \Longleftrightarrow 2ar+a -a^2< r^2+r \Longleftrightarrow a-r< r^2+a^2-2ar.$$ 
However, 
$a-r< (a-r)^2$ 
is always true since $r<a-1.$

For $(b)$, since $aq+a=d +q-r+a>d +q-1$  we can apply Lemma \ref{l. cokernel (a,b)-geproci}(b), thus
\[\dim(coker(\times \ell))=\sum_{i=0}^{q}  \binom{d+q+1-aq}{2}=(q+1)\cdot  \binom{r+1}{2}.\qedhere\]
\end{proof}

From Theorem \ref{t. WLP fails} it only  remains open to establish that, for  an $a\times a$ grid $X$ and $d = (a-1) q$, the algebra $R/\Lambda_{X,d}$  has the WLP. 

\begin{remark}\label{r. c=0}
    From the proof of Theorem \ref{t. WLP fails},
for $r=0$ and $d=(a-1)q$ we have
 $$\begin{array}{lcl}
    (a)\ \Delta h_A(d+q-1)&=&
\displaystyle \dfrac{aq-a^2q}{2}=-q\binom{a}{2}<0;\\[11pt]
(b)\ \dim \coker(\times \ell)&=&0.
\end{array} $$
Thus,  the map $\times \ell:[A]_{d+q-2}\to [A]_{d+q-1}$ is surjective. Furthermore,
$$\begin{array}{rl}
     \Delta h_A(d+q-2)=&h_A(d+q-2)- h_A(d+q-3)\\[11pt]
     =& \displaystyle\dim [R]_{d+q-2}-\dim [R]_{d+q-3}- \dim [I]_{d+q-2}+\dim [I]_{d+q-3} \\[11pt]
     =& \displaystyle\binom{d+q}{2}- a^2\cdot   \binom{q}{2}=\displaystyle \binom{aq}{2}- a^2\cdot   \binom{q}{2}\\[11pt]
    =& \displaystyle \dfrac{(aq)(aq-1)- a^2q(q-1)}{2}\\[11pt]
=& \displaystyle \dfrac{-aq+a^2q}{2}=q\binom{a}{2}.\\
\end{array} $$
Thus, the multiplication map $\times \ell:[A]_{d+q-3}\to [A]_{d+q-2}$ is expected to be injective.
\end{remark}

We recall the following fact.
\begin{remark}\label{r. surjectivity}
    Suppose that $M$ is a finite length  graded $R$-module and $\ell$ is a linear form. Assume that for some integer $n$, we have that  the map $\times \ell : [M]_n \rightarrow [M]_{n+1}$ is surjective and furthermore that $M$ has no minimal generator in degree $\geq n$. Then $\times \ell : [M]_t \rightarrow [M]_{t+1}$ is surjective for all $t \geq n$. This is because the module $M/(\ell\cdot M)$ is zero in degree $n+1$ and has no generators beyond this degree so it must be zero past this point. 
\end{remark}

\begin{theorem}\label{t. WLP holds}
Let $X$ be an $a\times a$ grid and $d = (a-1) q$, where $q\ge 1$. 
Set $I=\Lambda_{X,d}$ and $A = R/I$. Then
$A$ has the WLP. 
\end{theorem}
\begin{proof}Let $\ell$ be a general linear form.
From Remark \ref{r. c=0}, the map given by the multiplication by $\times \ell:A_{t-1}\to A_{t}$ is surjective for $t=d+q-1$, and hence for $t>d+q-1$.

We claim that  if $\times \ell:A_{t-1}\to A_{t}$ is injective for $t=d+q-2$ then it is injective for $t<d+q-2$ too. 
    
    Indeed, let $B$ the the $k$-dual of $A$ as a graded $R$-module. The fact (from Corollary \ref{c. no syzygies in low degrees}) that $A$ has no socle in degree $\leq d+q-2$ means that $B$ has no minimal generator in the corresponding range of components. Injectivity for $A$ becomes surjectivity for $B$ in the corresponding components. Thus our desired result follows from Remark \ref{r. surjectivity}.

Thus, to conclude the proof it is enough to show that the multiplication map  $\times \ell:[A]_{d+q-3}\to [A]_{d+q-2}$ is injective, i.e., we want to show that the cokernel of $\times \ell: [A]_{d+q-3} \rightarrow [A]_{d+q-2}$ has the expected dimension, that is (from Remark \ref{r. c=0}) $q\binom{a}{2}$. We use Lemma~\ref{l. cokernel (a,b)-geproci} to compute this dimension directly.
Indeed, since $d+q-2=aq-2=a(q-1)+ (a-2)< a(q-1)+ a$ from Lemma~\ref{l. cokernel (a,b)-geproci}(b) we have 
\[\dim coker(\times \ell)=\sum_{i=0}^{q-1}\binom{d+q-a(q-1)}{2}=q\binom{a}{2}.\qedhere \]
\end{proof}

\section{The non-Lefschetz locus}\label{sec: non-Lefschetz locus}

For an artinian graded algebra $R/I$ possessing the WLP, it is of interest to study the non-Lefschetz locus, i.e. the locus of linear forms for which multiplication  does not have maximal rank. This was introduced (in this context) in \cite{BMMN}. It was shown that the non-Lefschetz locus actually  has a scheme structure coming  from the ideal of maximal minors of a homogeneous matrix of linear forms. However, one can also study the  radical of this  ideal, giving the locus of forms for which maximal rank fails but ignoring any possible non-reduced structure. In this section we focus on this  latter point of view.

So, let $X$ be an $a \times b$ grid, $a \leq b$, and let $\{ \ell_1,\dots,\ell_{ab}\}$ be the linear forms dual to the points of $X$. Let $A = R/I = R/(\ell_1^d,\dots,\ell_{ab}^d)$.

There are (at first glance) two situations where we know that $A$ has the WLP, and in all the other cases we have either proved or given evidence that WLP fails. Namely, we know that if $a=b$ then $A$ has the WLP if and only if either $1 \leq d \leq a-1$ (Corollary \ref{low powers, square grid}) or $d$ is a multiple of $a-1$ (Theorem \ref{t. WLP fails} and Theorem \ref{t. WLP holds}). If $a < b$, we know that for $1 \leq d \leq a-1$, $A$ has the WLP (Corollary  \ref{low powers, square grid}) and furthermore that the ideal $I$ coincides with the ideal coming from a square grid. In all other cases we will conjecture that WLP fails (Conjecture~\ref{conj}). So without loss of generality  we can assume that $X$ is a square grid, but we will have to consider the cases $1 \leq d \leq a-1$ and $d = q(a-1)$ separately.

\begin{remark} \label{enough}
    We know that for a general linear form $\ell$, the multiplication $[A]_{t-1} \rightarrow [A]_t$ has maximal rank, and that the cokernel has dimension
    \[
    \dim [R/(I,\ell)]_t = \dim [I_{\bar X}^{m}]_t,
    \]
    where $\bar X$ is the projection of $X$ to a general plane and $m$ is a positive integer depending on $d$ and $t$. Furthermore, $\bar X$ is a complete intersection of type $(a,a)$. 

Since the Hilbert functions of any two complete intersections of type $(a,a)$ are the same, when we specialize $\ell$ to find the non-Lefschetz locus, we can focus on the projections (from a point dual to $\ell$) whose images are not complete intersections of type $(a,a)$. Thus, we focus on the ``non-CI locus", although we do not claim that these two loci are the same. (A projection might fail to be a complete intersection but still have the right Hilbert function). We just observe that the ``non-CI locus" contains the non-Lefschetz locus.

So we have an $a \times a$ grid $X$ consisting of $a^2$ points. Let us now write these as $\{P_{1,1},\dots,P_{a,a} \}$. These lie on a quadric $\mathcal Q$ and are cut out by $a$ lines $\lambda_1,\dots,\lambda_a$ in one ruling and $a$ lines $\mu_1,\dots,\mu_a$ in the other ruling. Let $\Lambda_{i,j}$ be the plane spanned by $\lambda_i$ and $\mu_j$. Note that 

\begin{itemize}

\item $\Lambda_{i,j}$ is the tangent plane to $\mathcal Q$ at $P_{i,j}$.

    \item $\Lambda_{i,j}$ contains all the points of $X$ on $\lambda_i$ and all the points of $X$ on $\mu_j$ and no other points of $X$.

    \item $\Lambda_{i,j}$ contains any line joining a point of $\lambda_i$ and a point of $\mu_j$.

    \item Every  line joining two points $P_{i,j}$ and $P_{p,q}$ of $X$ lies on either two of the planes (if $i \neq p$ and $j \neq q$) or $a$ of them (otherwise).
\end{itemize}

Putting it together, we observe that the projection $\pi_P(X)$ of $X$ from a point $P \in \PP^3$ is a complete intersection (namely the intersection of the $\pi_P(\lambda_i)$ and the $\pi_P(\mu_j)$) if and only if $P \notin \Lambda_{i,j}$ for any $1 \leq i \leq a, \ 1 \leq j \leq a$. So to find our non-Lefschetz locus, it suffices to look at points on the $\Lambda_{i,j}$.

\end{remark}

\begin{proposition}
    Let $X$ be an $a \times a$ grid and let $I$ be the ideal $(\ell_1^d,\dots,\ell_{a^2}^d)$ in the notation above. If $d \geq a$ we write $d = q(a-1)$.

    \begin{itemize}
        \item[(a)] If $d \leq a-1$ then the non-Lefschetz locus is empty.

         \item[(b)] If $d = 2(a-1)$ then the non-Lefschetz locus is the union  of the $2a$ grid lines and all other lines joining (exactly) two grid points.

        \item[(c)] If $d = q(a-1)$ for $q \geq 3$ then the non-Lefschetz locus is the union of the $a^2$ planes~$\Lambda_{i,j}$.

    \end{itemize}
\end{proposition}

\begin{proof}
    For (a), we saw that in this situation $R/I$ is a compressed Gorenstein algebra with even socle degree $2(a-1)$. Thus $R/I$ coincides with $R$ in the ``first half", i.e. up to degree $a-1$. Since $R$ is an integral domain, $\times \ell$ has no kernel in this range no matter what $\ell$ is. Since $R/I$ is Gorenstein, duality  provides surjectivity for the ``second half" of the Hilbert function. So from now on we can assume $q \geq 2$.

    For (b) and (c), recall that we have $d = q(a-1)$ and that we showed in Theorem \ref{t. WLP holds} that the peak for the Hilbert function  comes in degree $aq-2$.

We first study the behavior of projections $\pi_P$ from a general point $P$ of one of the planes $\Lambda_{i,j}$. Our goal will be to show that when $q = 2$, the linear form dual to $P$ gives a multiplication of maximal rank both from degree $aq-3$ to $aq-2$ (injectivity), and from degree $aq-2$ to $aq-1$ (surjectivity), while for $q \geq 3$ it fails maximal rank (we will only check the failure of surjectivity since that is enough).

Without loss of generality, let $P$ be a general point in the plane $\Lambda_{1,1}$ (the other planes behave in an identical way). The projection $\pi_P$ to a general plane maps the points $P_{1,j}$ and $P_{i,1}$ into a line $\overline{\lambda}$ and the other points into an $(a-1)\times (a-1)$ complete intersection of $\mathbb P^2$. 

We first consider the case $q \geq 3$. From Remark \ref{r. c=0} we know that $\Delta h_A(d+q-1)<0$, so it suffices to show that $\dim[I_{\pi_P(X)}^q]_{d+q-1}=\dim[I_{\pi_P(X)}^q]_{aq-1}>0.$ 
We proceed by induction.

The points $\pi_P(P_{i,j})$ with $i,j>1$ define a pencil of curves of degree $a-1$ in $\mathbb P^2$, so let $\gamma$ be the curve of such pencil containing $\pi_P(P_{1,1}).$
Then the curve  $\vartheta=\pi_P(\lambda_2\cup \cdots\cup \lambda_a\cup\mu_2\cup \cdots\cup \mu_a)\cup 2{\overline \lambda}\cup \gamma$ has degree $3a-1$ and
vanishes at the points of $\pi_P(X)$ with multiplicity 3.

If $q=3$ then we are done (we showed the existence of a curve of degree $3a-1 = aq-1)$.
Now assume $q>3.$ Since $\vartheta$ vanishes at each point of $\pi_P(X)$ to multiplicity 3, it is enough to show that there is a curve of degree $(aq-1) - (3a-1) = a(q-3)$ with multiplicity $q-3$ at each point of $\pi_P(X)$. If $q-3 = 1$, we have the curve $C = \pi_P(\lambda_2 \cup \dots \cup \lambda_a) \cup \bar{\lambda}$. If $q-3=2$, we take the curce $2C$, and in general, we have $(q-3)C$ and we are done.

We have shown that for $P \in  \Lambda_{i,j}$ a general point, and $q \geq 3$, $P$ is in the non-Lefschetz locus. By Remark \ref{enough}, this proves (c).

Finally we prove (b). So assume $d = 2(a-1)$. First let $P$ be a general point of $\Lambda_{1,1}$. We consider
\[
\dim [I_{\pi_P(X)}^2]_{2a-1}.
\]
Suppose $F$ were a curve in this  linear system. Every line $\pi_P(\lambda_i) \ (2 \leq i \leq a)$ and  every line $\pi_P(\mu_j) \ (2 \leq j \leq a)$, together with $\bar \lambda$, is a component of $F$. This curve vanishes to multiplicity 2 at every  point of $\pi_P(X)$ except $\pi_P(P_{1,1})$. But we have already reached degree $2a-1$, so this  dimension is 0. Thus $\Lambda_{1,1}$ (and similarly $\Lambda_{i,j}$) does not lie in the non-Lefschetz locus coming from surjectivity. 

However, we also  have to check injectivity for the dual to a general point in $\Lambda_{1,1}$. For a general projection a quick computation gives that the dimension of the cokernel of the multiplication is 
\[
\dim [I_{\pi_P(X)}]_{2a-2} = 2 \cdot \binom{a-2+2}{2} = 2 \binom{a}{2}.
\]
For a general $P \in \Lambda_{1,1}$, the image contains $2a-1$ points on $\bar \lambda$ so we seek the dimension of the degree $2a-3$ component of the ideal of a complete intersection of type $(a-1,a-1)$, which one quickly computes is also $2 \binom{a}{2}$. Since we saw that these two degrees are the only relevant ones to check WLP, this shows that $\Lambda_{i,j}$ is not contained in the non-Lefschetz locus.

Finally we have the case where $q = 2$ and $P$ is a general point on the line joining two points of $X$. We want to show that surjectivity fails from degree $2a-2$ to $2a-1$, so we have to show that $\dim [I_{\pi_P(X)}^2]_{2a-1} > 0$. 

First suppose $P$ is a general point on one of the $\lambda_i$ or $\mu_j$ (say $\lambda_1$). Then $\pi_P(X)$ lies on the curve $C = \pi_P(\mu_1 \cup \dots \cup \mu_a)$, which is a union of $a$ concurrent lines ($\lambda_1$ is collapsed to a point $A$). It is the union of the complete intersection of $C$ with $\pi_P(\lambda_2 \cup \dots \cup \lambda_a)$ and the point $A$. The union of these $2a-1$ lines are at least double at each point of  $\pi_P(X)$ so we are done.

Now suppose that $P$ is a general point on a line $\alpha$ joining two points of $X$ that are not on a common ruling line. Note that $\alpha$ contains exactly  two points of $X$, which are collapsed to the same point. Without loss of generality say $\alpha$ is the line joining $P_{1,2}$ and $P_{2,1}$. Let $E = \pi_P(P_{1,2}) = \pi_P(P_{2,1})$. Note that $\pi_P$ collapses $\mu_2$ and $\lambda_2$ to the same line, as well as collapsing $\lambda_1$ and $\mu_1$ to the same line. So $\pi_P(X)$ (as a set) is  a complete intersection $W$ of type $(a-1,a-1)$. Then $\dim[I_W^2]_{2a-1} = 3 \cdot 3 = 9$, so again surjectivity fails. 
\end{proof}

\section{Open questions and partial answers}\label{sec. conj non square grids}

In this section we propose some open questions coming from the work in this paper, together with some discussion. 

\begin{enumerate}
    \item Is it true that in the non-square case, all remaining cases fail the WLP? More precisely, we formulate the following conjecture.  

\begin{conjecture}\label{conj}
Let $X$ be an $a\times b$ grid, with $b> a\ge 2$. For any integer $d\ge a$, set $I=\Lambda_{X,d}$ and $A=R/I$. Then $A$ fails to have the WLP.
\end{conjecture}

\begin{remark}\label{r. dim cokernel}
    Let $q',r'$ be the integers such that $d=(a-1)q'+r'$ where $q'\ge 1$ and $1\le r'\le a-1$. 
    Let $\ell$ be a general linear form. 
Experimentally the the map
    $$\times  \ell: A_{d+q'-2}\to A_{d+q'-1}$$
fails to have maximal rank.
From Lemma \ref{l. cokernel (a,b)-geproci} we have 
    $$\dim coker(\times \ell)=\sum_{i=0}^{q'}\binom{d+q'+1-aq'-(b-a)i}{2}=\sum_{i=0}^{q'}\binom{r'+1-(b-a)i}{2}>0.$$
\end{remark}

\begin{remark}
    To prove the conjecture one has to show that 
$$\Delta h_A(d+q'-1)<\sum_{i=0}^{q'}\binom{r'+1-(b-a)i}{2}.$$ 
    
Let $q,r$ such that $d=(b-1)q+r$ with $0\le r<b-1$.
 Then we have 
$$d=(b-1)q+r=(a-1)q'+r'$$
where $1\le r'\le a-1$ thus $$d-1=(a-1)q'+(r'-1)$$  
with $0\le r'-1<a-1$. Therefore we get 
$q\le q'$
If $d+q'-1 \le d+q-1$, i.e. $q=q'$ then we are in the range of Corollary \ref{c. no syzygies in low degrees}.  
Moreover, if $q\neq q'$ then it seems true that $\Delta h_A(d+q'-1)\le 0$.
\end{remark}

\begin{remark}\label{r. Macaulay + ses}
    By Macaulay duality 
$$\Delta h_A(d+q'-1)=h_A(d+q'-1)- h_A(d+q'-2)=\dim[I_X^{(q')}]_{d+q'-1}-\dim[I_X^{(q'-1)}]_{d+q'-2}.$$
Moreover from the short exact sequence in \eqref{new ses} we get
$$\dim[I_X^{(\alpha)}]_{t}=\dim[I_X^{(\alpha-1)}]_{t-2}+\dim[I_Z^{(\alpha)}]_{t}
$$
where $Z$ is a set of points in $\mathbb P^1\times \mathbb P^1$ whose Segre embedding is $X.$
The above formula can be applied recursively to our case to get
$$\dim[I_X^{(q')}]_{d+q'-1}=\dim[I_X]_{d-q'+1}+\dim[I_Z^{(2)}]_{d-q'+3}+\cdots+ \dim[I_Z^{(q'-1)}]_{d+q'-3}+\dim[I_Z^{(q')}]_{d+q'-1} $$
and an  analogous formula for $\dim[I_X^{(q'-1)}]_{d+q'-2}$. 
The Hilbert function of $k[\mathbb P^1\times \mathbb P^1]/I_Z^{(t)}$ is known, see for instance Corollary 2.3 in \cite{GV}, but it is not in a closed formula, which makes the computation of the general case arduous.
\end{remark}
We show how to prove the failure of WLP for the algebra coming from a non-square grid in  specific cases following the above remarks.
\begin{example} 
Let $X$ be a $3\times 6$ grid and $d=5$ and let $\ell$ be a general linear form. Let $I=(\ell^5\ |\ \ell^{\vee}\in X)$ and $A=R/I.$

Since $5=1\cdot 3+2$, according to the conjecture we have to look at $\times \ell : A_{5}\to A_{6}.$
From Remark \ref{r. dim cokernel} the dimension of the cokernel of this map is $1$.
However from Remark \ref{r. Macaulay + ses} we can compute
$$h_A(6)-h_A(5)=\dim [I_X^{(2)}]_6-\dim [I_X]_5=\dim [I_X]_4+\dim [I_Z^{(2)}]_6-\dim [I_X]_5=17+10-38=-11.$$
This shows that $A$ fails the WLP.
\end{example}

\item One of the key facts about grids that was used in this paper was that a general projection is a complete intersection. But grids form a small subset of all sets in $\PP^3$ with this  property (see for instance \cite{POLITUS1}). Most such sets do not lie on a quadric surface, so much of our machinery does not apply. Nevertheless, the following is an interesting problem.

\begin{question}
 Let $X$ be an $(a,b)$-geproci set. When does  $R/\Lambda_{X,d}$ have   the WLP?
\end{question}

The question above is posed in  its broadest generality. The answer might depend on the combinatorics of $X$. 
    Partial information to address the study of the WLP are contained in Lemma \ref{l. cokernel (a,b)-geproci}. In this case, what has to be investigated is the number of conditions imposed by the scheme of fat points $X^{(m)}$ to the vector space $R_{t}$ for suitable integers $m$ and $t$.  Experimental evidence suggests strong connections between a more general $(a,b)$-geproci set and an $a \times b$ grid, from the point of view of the WLP for the corresponding algebras $R/\Lambda_{X,d}$.

\medskip

\item The advantage of studying grids in this paper is that the general projection is a complete intersection, so the Hilbert functions of the powers are easy to compute. However, the Hilbert function of the symbolic power of $X$ was not so easy to compute. In this question we reverse the roles.

\begin{question}
    If $X \subset \PP^3$ is a complete intersection, can we say anything  about the WLP for $R/\Lambda_{X,d}$?
\end{question}

\noindent Now we know very little about the general projection of $X$, but the Hilbert function of the symbolic powers of $I_X$ are easy to compute.

\medskip

\item The Strong Lefschetz Property (SLP) for an artinian algebra $A$ is defined by the property that not only does $\times \ell$ have maximal rank in all degrees for a general linear form $\ell$ (which  is WLP), but in addition $\times \ell^k$ does as well. 
\begin{question}
    In all of the above situations, does $R/ \Lambda_{X,d}$ have the SLP?
\end{question}

\medskip

\item Finally, one can ask if the methods of this paper can be a starting point to studying ideals generated by {\it mixed} powers of linear forms when the forms again are dual to the points on a grid. 
\medskip
 \item For a set of points $X$ in $\mathbb P^n$ is the behavior of the WLP of $\Lambda_{X,d}$ eventually stable? We formalize this problem in the following question.  \begin{question}Fix $n$ and
    let $X$ be a set of points in $\mathbb P^n$. 
    Construct the sequence $B_X=(b_1, b_2, b_3, \ldots)$
    as follows: $b_d=1$ if $\Lambda_{X,d}$ has the WLP, otherwise it is 0.
    Also define from $B_X$ the corresponding real number $b_X=0.b_1b_2b_3\ldots$.  
    Is $b_X$ a rational number? 
    Which numbers $b$ constructed from a sequence of $0$ and $1$ as above are such that $b=b_X$ for some set of points in $\mathbb P^n?$    
\end{question}
\noindent For any $X\subseteq \mathbb P^2$, from \cite{SS1}, we have that $b_X=0.\overline{1}.$  The known results, for instance on general points and grids, show that in $\mathbb P^3$ there are more possibilities. Thus, when do two different sets of points $X$ and $X'$ have $b_X=b_{X'}$?
\end{enumerate}


\begin{thebibliography}{99}
\bibitem{BL} M. Boij and S. Lundqvist, {\it A classification of the weak Lefschetz property for almost complete intersections generated by uniform powers of general linear forms}, Algebra Number Theory 17 (2023), 111--126.

\bibitem{BMMN}  M. Boij, J. Migliore, R.M. Mir\'o-Roig and U. Nagel, {\it The non-Lefschetz locus}, J. Algebra 505 (2018), 288--320.

\bibitem{POLITUS1} L. Chiantini, \L\ Farnik, G. Favacchio, B. Harbourne, J. Migliore, J. Szpond and T. Szemberg, {\it Configurations of points in projective space and their projections}, preprint 2022.

\bibitem{POLITUS4} L. Chiantini, \L\ Farnik, G. Favacchio, B. Harbourne, J. Migliore, J. Szpond and T. Szemberg, {\it Geprofi sets of points in $\PP^4$}, work in progress.

\bibitem{CM} L. Chiantini and J. Migliore, {\it Sets of points which project to complete intersections, and
unexpected cones}, Trans. Amer. Math. Soc., 374 (2021), pp. 2581–2607. With an appendix by A. Bernardi, L. Chiantini, G. Denham, G. Favacchio, B. Harbourne, J. Migliore, T. Szemberg and J. Szpond.

\bibitem{codim2} S. Cooper, G. Fatabbi, E. Guardo, A. Lorenzini, J. Migliore, U. Nagel, A. Seceleanu, J. Szpond, and A. Van Tuyl, {\it Symbolic powers of codimension two Cohen-Macaulay ideals}, Comm. Algebra 48 (2020), 4663--4680.

\bibitem{GMR} S. Giuffrida, R. Maggioni, A. Ragusa, {\it On the postulation of 0-dimensional subschemes on a smooth quadric.} Pacific Journal of Mathematics, 155:2  (1992),  251--282.

\bibitem{GV-book} E. Guardo and A. Van Tuyl, {\it Arithmetically Cohen-Macaulay sets of points in $\mathbb P^1\times \mathbb P^1$}, SpringerBriefs in Mathematics, Springer, Cham, 2015.

\bibitem{GV} E. Guardo and A. Van Tuyl, {\it Powers of complete intersections: graded Betti numbers and applications}, Illinois J. Math. 49 (2005), 265--279.

\bibitem{HSS} B. Harbourne, H. Schenck and A. Seceleanu,  {\it Inverse systems, Gelfand–Tsetlin patterns and the weak
Lefschetz property}, J. Lond. Math. Soc. (2) 84:3 (2011), 712--730.

\bibitem{HMNT} B. Harbourne, J. Migliore, U. Nagel and Z. Teitler, {\it Unexpected hypersurfaces and where to find them}, Michigan Math. J. 70 (2021), 301--339.

\bibitem{I.compressed} A. Iarrobino,  {\it Compressed algebras: Artin algebras having given socle degrees and maximal length.} Transactions of the American Mathematical Society. 1984;285(1):337--78.

\bibitem{IK} A. Iarrobino and V. Kanev, ``Power sums, Gorenstein algebras, and determinantal loci", Lecture Notes in Mathematics, 1721. Springer-Verlag, Berlin, 1999.

\bibitem{MM} J. Migliore and R. M. Mir\'o-Roig, {\it On the Strong Lefschetz problem for uniform powers of general linear forms in $k[x,y,z]$}, Proc. Amer. Math. Soc. 146 (2018), 507--523.

\bibitem{MMN} J. Migliore, R. M. Mir\'o-Roig and U. Nagel, {\it On the Weak Lefschetz Property for powers of linear forms}, Algebra \& Number Theory 6 (2012), 487--526.

\bibitem{MN} J. Migliore and U. Nagel, {\it The Lefschetz question for ideals generated by  powers of linear forms in few variables} J. Comm. Algebra 13 (2021), 381--405.

\bibitem{MR} R. M. Miró-Roig, {\it Harbourne, Schenck and Seceleanu’s conjecture}, J. Algebra 462 (2016), 54--66.

\bibitem{MT} R. M. Miró-Roig and Q. H. Tran, {\it On the weak Lefschetz property for almost complete intersections
generated by uniform powers of general linear forms}, J. Algebra 551 (2020), 209--231.

\bibitem{NT} U. Nagel and W. Trok,  {\it Interpolation and the weak Lefschetz property}, Trans. Amer. Math. Soc. 372:12
(2019), 8849--8870.

\bibitem{SS1} {H. Schenck and A. Seceleanu,} {\it The weak Lefschetz property and powers of linear forms in $\mathbb K [x,y,z]$}, Proc. Amer. Math. Soc. 138 (2010), 2335--2339.
Proc. Amer. Math. Soc.

\bibitem {SS2} {H. Schenck and A. Seceleanu,}  {\it The Weak Lefschetz Property, inverse systems and fat points}, preprint 2010.

\bibitem{stanley} R. Stanley, {\it Weyl groups, the hard Lefschetz theorem, and the Sperner property}, SIAM J. Algebraic Discrete Methods 1 (1980), no. 2, 168--184.

\bibitem{watanabe} J. Watanabe, {\it The Dilworth number of Artinian rings and finite posets with rank function}, Commutative algebra and combinatorics (Kyoto, 1985), Adv. Stud. Pure Math., vol. 11, North-Holland, Amsterdam, 1987, pp. 303--312. 

\bibitem{ZS} O. Zariski and P. Samuel, ``Commutative algebra.'' Vol. II, The University Series in Higher Mathematics, D. Van Nostrand Co., Inc., Princeton, N. J.-Toronto-London- New York, 1960. 

\end{thebibliography}
\end{document}